\documentclass[11pt]{amsart}
\usepackage[T1]{fontenc}
\usepackage[utf8]{inputenc}
\usepackage{lmodern}
\usepackage{mathrsfs}
\usepackage{amsfonts}
\usepackage{latexsym,amsmath,amssymb}

 \renewcommand{\epsilon}{\varepsilon}

\newtheorem{theorem}{Theorem}[section]

 \newtheorem{lemma}[theorem]{Lemma}

 \newtheorem{proposition}[theorem]{proposition}
 
\newtheorem{deff}[theorem]{Definition}
 \newtheorem{rem}[theorem]{Remark}
 \newcommand{\bth}{\begin{theorem}}
 \newcommand{\ble}{\begin{lemma}}
 \newcommand{\bcor}{\begin{corr}}
 \newcommand{\bdeff}{\begin{deff}}
 \newcommand{\bprop}{\begin{proposition}}
 \newcommand{\ele}{\end{lemma}}
 \newcommand{\ecor}{\end{corr}}
 \newcommand{\edeff}{\end{deff}}
 
 \newcommand{\eprop}{\end{proposition}}

 \renewcommand{\Pi}{\varPi}

 \renewcommand{\epsilon}{\varepsilon}

\numberwithin{equation}{section}

\thanks{The authors were supported  by National Science Foundation of China(No.11771103) and  Guangxi Natural Science Foundation(No.2017GXNSFFA198017).}

\title
[Nonlinear second boundary conditions ]{On the second boundary value problem for special Lagrangian curvature potential  equation}
 \author{Sitong Li}
\address{School of Mathematics and Statistics, Guangxi Normal University,
Guilin, Guangxi 541004, People's Republic of China}
\email{2019011182@stu.gxnu.edu.cn}
\author{Rongli Huang}
\address{School of Mathematics and Statistics, Guangxi Normal University,
Guilin, Guangxi 541004, People's Republic of China}
\email{ronglihuangmath@gxnu.edu.cn}
\date{}

\begin{document}
\maketitle
\begin{abstract}
This is a sequel to \cite{HO} and \cite{CHB}, which study the second boundary  problem for special Lagrangian curvature potential equation. As consequences, we obtain the existence and uniqueness of the smooth
uniformly convex solution by the method of continuity through carrying out a-priori estimate on the solutions.

\end{abstract}

\let\thefootnote\relax

\section{Introduction}

The first author and his coauthors  were concerned with
 special Lagrangian  equations
\begin{equation}\label{e1.1}
F_\tau(\lambda(D^2 u))=c,\quad
 \mathrm{in}\quad \Omega,
\end{equation}
associated with the second boundary value problem
\begin{equation}\label{e1.2}
Du(\Omega)=\tilde{\Omega},
\end{equation}
for given $F_{\tau}$,  $\Omega$  and  $\tilde{\Omega}$. Specifically, $\Omega$, $\tilde{\Omega}$  are uniformly convex bounded
domains with smooth boundary in $\mathbb{R}^{n}$ and
\begin{equation*}
F_{\tau}(\lambda):=\left\{ \begin{aligned}
&\frac{1}{2}\sum_{i=1}^n\ln\lambda_{i}, &&\tau=0, \\
& \frac{\sqrt{a^2+1}}{2b}\sum_{i=1}^n\ln\frac{\lambda_{i}+a-b}{\lambda_{i}+a+b},  &&0<\tau<\frac{\pi}{4},\\
& -\sqrt{2}\sum_{i=1}^n\frac{1}{1+\lambda_{i}}, &&\tau=\frac{\pi}{4},\\
& \frac{\sqrt{a^2+1}}{b}\sum_{i=1}^n\arctan\frac{\lambda_{i}+a-b}{\lambda_{i}+a+b},  \ \ &&\frac{\pi}{4}<\tau<\frac{\pi}{2},\\
& \sum_{i=1}^n\arctan\lambda_{i}, &&\tau=\frac{\pi}{2}.
\end{aligned} \right.
\end{equation*}
Here $a=\cot \tau$, $b=\sqrt{|\cot^2\tau-1|}$,  $x=(x_{1},x_{2},\cdots,x_{n})$, $u=u(x)$ and $\lambda(D^2 u)=(\lambda_1,\cdots, \lambda_n)$ are the eigenvalues of Hessian matrix $D^2 u$ according to $x$. The equations (\ref{e1.1}) comes from \cite{MW}, where Warren established the calibration theory for spacelike Lagrangian submanifolds in $(\mathbb R^{2n},g_\tau)$ (for general calibration theory for Lagrangian submanifolds in Riemannain and pseudo-Riemannain manifolds, see Harvey-Lawson \cite{HL} and Mealy \cite{Mealy}).

We present here the main results regarding the sovability of solutions to (\ref{e1.1})-(\ref{e1.2}).
 Brendle and Warren \cite{SM} proved the existence and uniqueness of the solution by the elliptic method in the case of $\tau=\frac{\pi}{2}$.
 The first author \cite{HR} considered the longtime existence and convergence of the solutions to the second boundary problem of  Lagrangian mean curvature flow and also obtained existence results in this case. As far as $\tau=0$ is concerned, the equation (\ref{e1.1}) can be transformed into  Monge-Amp\`{e}re equation. Some results had already been given in the literature for the second boundary value problem of Monge-Amp\`{e}re equation.
 Delano\"{e} \cite{P} obtained a unique smooth solution in dimension 2 if both domains are uniformly convex. Later his result was extended to any dimension by Caffarelli \cite{L} and Urbas \cite{JU}. Using the parabolic method, Schn\"{u}rer and Smoczyk \cite{OK} also arrived at  the existence of solutions to (\ref{e1.1}) for $\tau=0$. For further studies on  the rest of $0\leq\tau\leq\frac{\pi}{2}$, see  \cite{HO}, \cite{HRY}   \cite{CHY} and the references therein.

The aim of this paper is to study special Lagrangian curvature potential equation
\begin{equation}\label{e1.3}
\sum_{i=1}^n\arctan\kappa_{i}=c,\quad
 \mathrm{in}\quad \Omega,
\end{equation}
in conjunction with the  second boundary value problem
\begin{equation}\label{e1.4}
Du(\Omega)=\tilde{\Omega},
\end{equation}
where $(\kappa_1,\cdots, \kappa_n)$ are the principal curvatures of the graph $\Gamma=\{(x,u(x))\mid x\in \Omega\}$
and $c$ is a constant to be determined, and $\Omega$, $\tilde{\Omega}$  are uniformly convex bounded
domains with smooth boundary in $\mathbb{R}^{n}$. It shows that a very nice geometric interpretation of special Lagrangian curvature operator
$F(\kappa_{1},\cdots, \kappa_{n})\triangleq\sum_{i=1}^n\arctan\kappa_{i}$ in \cite{G}.
The background of the so-call special Lagrangian curvature potential  equation (\ref{e1.3}) can be seen in the literature \cite{G} \cite{HL1}.

Now we state our main theorem.

\begin{theorem}\label{t1.1}
 Suppose that  $\Omega$, $\tilde{\Omega}$  are uniformly convex bounded
domains with smooth boundary in $\mathbb{R}^{n}$. Then there exist a uniformly convex solution $u\in C^{\infty}(\bar{\Omega})$ and a unique constant $c$ solving (\ref{e1.3})-(\ref{e1.4}), and $u$ is unique up to a constant.
\end{theorem}

To obtain the existence result  we use the method of continuity by carrying out a-priori estimate on the solutions to (\ref{e1.3})-(\ref{e1.4}).

The rest of this article is organized as follows. The next section is to present the structure conditions for the operator $F$ with the related operator $G$ and the fundamental formulas  for the principal curvatures of the graph $\Gamma$ in $\mathbb{R}^{n}$. By the preliminary knowledge we implement the strictly oblique estimate in section 3 and then in section 4 we obtain the $C^2$ estimate by considering the operator on manifold as same as J. Urbas'work \cite{J}.  Finally we give the proof of the main theorem by the continuity method in section 5.
\section{Preliminary}

We begin by stating some facts about the special Lagrangian curvature operator
$$F(\kappa_{1},\cdots, \kappa_{n}):=\sum_{i=1}^n\arctan\kappa_{i},$$
and $$\Gamma^+_n:=\left\{(\kappa_1,\cdots,\kappa_n)\in \mathbb{R}^n:\kappa_i>0,\ i=1,\cdots,n\right\}.$$
These properties are trivial and can be found in \cite{HO}.

\begin{lemma}\label{l2.1}
Based on the above definition, the following properties are established.\\
(i) $F$ is  a smooth symmetric function defined on $\overline {{\Gamma}^+_n}$ and satisfying
$$-\infty<F(0,\cdots,0)<F(+\infty,\cdots,+\infty)<+\infty,$$
$$\frac{\partial F}{\partial \kappa_i}>0,\ \ 1\leq i\leq n\ \  \text{on}\ \  \Gamma^+_n,$$
and
$$\left(\frac{\partial^2 F}{\partial \kappa_i\partial \kappa_j}\right)\leq 0\ \  \text{on}\ \  \overline {{\Gamma}^+_n}.$$
(ii) For any $(\mu_1,\cdots,\mu_n)\in {\Gamma}^+_n$, denote
$$\kappa_i=\frac{1}{\mu_i},\ \ 1\leq i\leq n,$$
and
$$\tilde F (\mu_1,\cdots,\mu_n):=-F(\kappa_1,\cdots,\kappa_n).$$
Then
$$\left(\frac{\partial^2 \tilde F}{\partial \mu_i\partial \mu_j}\right)\leq 0\ \  \text{on}\ \  \overline {{\Gamma}^+_n}.$$
(iii) For any $s_1>0$, $s_2>0$, define
$$\Gamma^{+}_{]s_1,s_2[}=\{(\kappa_{1},\cdots, \kappa_{n})\in {\Gamma}^+_n:0\leq\min_{1\leq i\leq n}\kappa_i\leq s_1,\ \max_{1\leq i\leq n}\kappa_i\geq s_2\}.$$
Then there exist positive constants $\Lambda_1$ and $\Lambda_2$, depending  on $s_1$ and $s_2$, such that for any $(\kappa_{1},\cdots, \kappa_{n})\in \Gamma^{+}_{]s_1,s_2[}$,
 $$\Lambda_1\leq\sum^{n}_{i=1}\frac{\partial F}{\partial \kappa_{i}}\leq \Lambda_2,$$
and
$$\Lambda_1\leq\sum^{n}_{i=1}\frac{\partial F}{\partial \kappa_{i}}\kappa^{2}_{i}\leq \Lambda_2.$$
\end{lemma}
It follows from \cite{GT} and \cite{LLJ}   that  we can state various geometric quantities associated with the graph of  $u\in C^{2}(\Omega)$.
In the coordinate system Latin indices range from 1 to $n$ and indicate quantities in the graph. We use the Einstein summation convention,
if the indices are different from 1 and n. $u_{i}, u_{ij},u_{ijk} \cdots$ denote the  all derivatives of $u$ according to $x_{i},x_{j},x_{k}\cdots$.
 The induced metric on  $$\Gamma\triangleq\{(x,u(x))\mid x\in \Omega\}$$ is given by
\begin{equation*}
g_{ij}=\delta_{ij}+u_{i}u_{j}
\end{equation*}
and its inverse is
\begin{equation*}
g^{ij}=\delta_{ij}-\frac{u_{i}u_{j}}{1+|Du|^{2}}.
\end{equation*}
The second fundamental form is denoted as
\begin{equation*}
h_{ij}=\frac{u_{ij}}{\sqrt{1+|Du|^{2}}}.
\end{equation*}
The principal curvatures of $\Gamma$ are the eigenvalues of the second fundamental form
$h_{ij}$ relative to $g_{ij}$, i.e, the eigenvalues of the mixed tensor $h^{j}_{i}\equiv h_{ik}g^{kj}$.
By \cite{LLJ} we remark that they are the eigenvalues of the symmetric matrix
\begin{equation}\label{e2.7}
a_{ij}=\frac{1}{v}b^{ik}u_{kl}b^{lj},
\end{equation}
where $v=\sqrt{1+|Du|^{2}}$ and $b^{ij}$ is the positive square root of $g^{ij}$ taking the form
\begin{equation*}
b^{ij}=\delta_{ij}-\frac{u_{i}u_{j}}{v(1+v)}.
\end{equation*}
Explicitly we have shown
\begin{equation}\label{e2.8}
a_{ij}=\frac{1}{v}\left\{u_{ij}-\frac{u_{i}u_{l}u_{jl}}{v(1+v)}-\frac{u_{j}u_{l}u_{il}}{v(1+v)}
+\frac{u_{i}u_{j}u_{k}u_{l}u_{kl}}{v^{2}(1+v)^{2}}\right\}.
\end{equation}
Then by (\ref{e2.7}) one can deduce that
\begin{equation}\label{e2.9}
u_{ij}=vb_{ik}a_{kl}b_{lj}
\end{equation}
where $b_{ij}$  is the inverse of $b^{ij}$ expressed as
\begin{equation*}
b_{ij}=\delta_{ij}+\frac{u_{i}u_{j}}{1+v}.
\end{equation*}

Denote $A=[a_{ij}]$ and $F[A]=\sum_{i=1}^n\arctan\kappa_{i}$, where $(\kappa_1,\cdots, \kappa_n)$ are the  eigenvalues of
the symmetric matrix $[a_{ij}]$. Then the properties of the operator $F$ are reflected in Lemma \ref{l2.1}.
It follows from Lemma \ref{l2.1} (i) that we can show that
\begin{equation*}
F_{ij}[A]\xi_{i}\xi_{j}>0 \quad \mathrm{for}\quad \mathrm{all}\quad \xi\in \mathbb{R}^{n}-\{0\}
\end{equation*}
where
\begin{equation*}
F_{ij}[A]=\frac{\partial F[A]}{\partial a_{ij}}.
\end{equation*}
From \cite{J} we see that $[F_{ij}]$ diagonal if A is diagonal, and in this case
\begin{equation*}
[F_{ij}]=\mathrm{diag}\left(\frac{\partial F}{\partial \kappa_{1}},\cdots, \frac{\partial F}{\partial \kappa_{n}}\right).
\end{equation*}
If $u$ is convex, by (\ref{e2.7}) we deduce that the  eigenvalues of the  matrix $[a_{ij}]$ must be in $\overline {{\Gamma}^+_n}$.
Then Lemma \ref{l2.1} (i) implies  that
\begin{equation*}
F_{ij,kl}[A]\eta_{ij}\eta_{kl}\leq 0
\end{equation*}
for any real symmetric matrix $[\eta_{ij}]$, where
\begin{equation*}
F_{ij,kl}[A]=\frac{\partial^{2}F[A]}{\partial a_{ij}\partial a_{kl}}.
\end{equation*}
According to equation (\ref{e1.3}) we consider the fully nonlinear elliptic differential equation of the type
\begin{equation}\label{e2.10}
G(Du,D^{2}u)=F[A]=c.
\end{equation}
As in \cite{J}, differentiating this once we have
\begin{equation}\label{e2.11}
G_{ij}u_{ijk}+G_{i}u_{ik}=0
\end{equation}
where we use the notation
\begin{equation*}
G_{ij}=\frac{\partial G}{\partial r_{ij}},\quad G_{i}=\frac{\partial G}{\partial p_{i}}
\end{equation*}
with $r$ and $p$  representing for the second derivative and gradient variables respectively. So as to prove the strict obliqueness estimate for the problem (\ref{e1.3})-(\ref{e1.4}), we need to recall some expressions  for the derivatives of $G$ from a simple calculation already done in \cite{J}. In fact, there holds
\begin{equation}\label{e2.12}
G_{ij}=F_{kl}\frac{\partial a_{kl}}{\partial r_{ij}}=\frac{1}{v}b^{ik}F_{kl}b^{lj}
\end{equation}
and
\begin{equation*}
G_{i}=F_{kl}\frac{\partial a_{kl}}{\partial p_{i}}=F_{kl}\frac{\partial}{\partial p_{i}}\left(\frac{1}{v}b^{kp}b^{ql}\right)u_{pq}.
\end{equation*}
A simple calculation yields
\begin{equation*}
G_{i}=-\frac{u_{i}}{v^{2}}F_{kl}a_{kl}-\frac{2}{v}F_{kl}a_{lm}b^{ik}u_{m}.
\end{equation*}

 The explicit expression for $\mathcal{T}_{G}=\sum^{n}_{i=1}G_{ii}$ is the trace of a product of three matrices by (\ref{e2.12}), so it is invariant under orthogonal transformations. Hence, to compute $\mathcal{T}_{G}$, we may assume for now that $[a_{ij}]$ is diagonal. With the aid of  (\ref{e1.4}), we obtain that $Du$ is bounded and the eigenvalues of $[b^{ij}]$ are bounded between two controlled positive constants. Since (\ref{e2.12}), it follows that there exist positive constants $\sigma_{1}$,$\sigma_{2}$ depending only on the least upper bound of $|Du|$ in  $\Omega$, such that
\begin{equation}\label{e2.15}
\sigma_{1}\mathcal{T}\leq\mathcal{T}_{G}\leq\sigma_{2}\mathcal{T}
\end{equation}
where $\mathcal{T}=\sum^{n}_{i=1}F_{ii}$.

For the same reason, by making use of Lemma \ref{l2.1} (iii), (\ref{e2.9}) and (\ref{e2.12}),  there exist two positive constants
$\sigma_{3}$,$\sigma_{4}$ depending only on the  upper bound of $|Du|$ in  $\Omega$, such that
\begin{equation}\label{e2.16}
\sigma_{3}\sum^{n}_{i=1}\frac{\partial F}{\partial \kappa_{i}}\kappa^{2}_{i}\leq \sum^{n}_{i=1}\frac{\partial G}{\partial \lambda_{i}}\lambda^{2}_{i}\leq \sigma_{4}\sum^{n}_{i=1}\frac{\partial F}{\partial \kappa_{i}}\kappa^{2}_{i}
\end{equation}
where $(\lambda_1,\cdots, \lambda_n)$ are the eigenvalues of Hessian matrix $D^2 u$ according to $x$.

The concavity of $F$ and the positive definiteness of $[F_{ij}a_{ij}]$ imply that$[F_{ij}a_{ij}]$ is bounded, i.e.
\begin{equation}\label{e2.17}
F_{ij}a_{ij}=\sum^{n}_{i=1}F_{i}\kappa_{i}\leq F(\kappa_{1},\ldots,\kappa_{n})\leq\frac{n\pi}{2}.
\end{equation}
Thus $G_{i}$ is bounded.

In order to get the strict obliqueness estimate, we will use the Legendre transform of $u$ which is the convex function $u^{\ast}$ on $\tilde{\Omega}=Du(\Omega)$ defined by
\begin{equation*}
u^{\ast}(y)=x\cdot Du(x)-u(x)
\end{equation*}
and
\begin{equation*}
y=Du(x).
\end{equation*}
It follows that
\begin{equation*}
\frac{\partial u^{\ast}}{\partial y_{i}}=x_{i},\frac{\partial ^{2}u^{\ast}}{\partial y_{i}\partial y_{j}}=u^{ij}(x)
\end{equation*}
where $[u^{ij}]=[D^{2}u]^{-1}$.\\
 In view of (\ref{e1.4}), (\ref{e2.10}), the function $u^{\ast}$ in $\tilde{\Omega}$ satisfies that
\begin{equation}\label{e2.20}
G^{\ast}(y, D^{2}u^{\ast})\triangleq-G(y,[D^{2}u^{\ast}]^{-1})+\frac{n\pi}{2}=\frac{n\pi}{2}-c
\end{equation}
and
\begin{equation}\label{e2.21}
Du^{\ast}(\tilde{\Omega})=\Omega.
\end{equation}
Hence, we rewrite equation (\ref{e2.20}) as
\begin{equation}\label{e2.22}
F^{\ast}[a^{\ast}_{ij}]\triangleq F(\mu_{1},\ldots,\mu_{n})=-F(\mu_{1}^{-1},\ldots,\mu_{n}^{-1})+\frac{n\pi}{2}=\frac{n\pi}{2}-c
\end{equation}
where $\mu_{1},\ldots,\mu_{n}$ are the eigenvalues of the matrix $[a^{\ast}_{ij}]$ given by
\begin{equation*}
a^{\ast}_{ij}=\sqrt{1+|y|^{2}}b^{\ast}_{ik}u_{kl}^{\ast}b^{\ast}_{lj}
\end{equation*}
where
\begin{equation*}
b^{\ast}_{ij}=\delta_{ij}+\frac{y_{i}y_{j}}{1+\sqrt{1+|y|^{2}}}.
\end{equation*}
The inverse matrix $[b^{\ast ij}]$ of $[b^{\ast}_{ij}]$ is given by
\begin{equation*}
b^{\ast ij}=\delta_{ij}-\frac{y_{i}y_{j}}{\sqrt{1+|y|^{2}}(1+\sqrt{1+|y|^{2}})}.
\end{equation*}
and
\begin{equation}\label{e2.21aa}
G^{\ast}(y, D^{2}u^{\ast})=F^{\ast}[a^{\ast}_{ij}]=\frac{n\pi}{2}-c.
\end{equation}
In the following we need to examine the bound properties of the coefficients for the linearized operator
according to (\ref{e2.20}). On account of $y\in\tilde{\Omega}$, an argument similar to that used above show that there exist two controlled positive constants $\sigma_{5}, \sigma_{6}$, such that
\begin{equation*}
\sigma_{5}\mathcal{T}^{\ast}\leq\mathcal{T}^{\ast}_{G}\leq\sigma_{6}\mathcal{T}^{\ast}
\end{equation*}
where $\mathcal{T}^{\ast}=\sum^{n}_{i=1}F^{\ast}_{ii}$ and $\mathcal{T}^{\ast}_{G}=\sum^{n}_{i=1}G^{\ast}_{ii}$. Taking derivatives to (\ref{e2.20}), we conclude that
\begin{equation}\label{e2.27}
G^{\ast}_{ij}u_{ijk}^{\ast}+G^{\ast}_{y_{k}}=0
\end{equation}
where
\begin{equation*}
G^{\ast}_{ij}=\sqrt{1+|y|^{2}}b^{\ast}_{ik}F^{*}_{kl}b^{*}_{lj}
\end{equation*}
and
\begin{equation*}
\begin{aligned}
G^{*}_{y_{k}}=\frac{\partial G^{*}}{\partial y_{k}}&=F^{*}_{il}\frac{\partial}{\partial y_{k}}\left(\sqrt{1+|y|^{2}}b^{*}_{ip}b^{*}_{ql}\right)u_{pq}^{*}\\
&=\frac{y_{k}}{1+|y|^{2}}F^{*}_{il}a^{*}_{il}+2F^{*}_{il}a^{*}_{lm}(D_{k}b^{*}_{ip})b^{*pm}.\\
\end{aligned}
\end{equation*}
By making use of $F^{*}=F$, as same as (\ref{e2.17}), we can show that $G^{*}_{y_{i}}$ is also bounded.
\vspace{5mm}
\section{The strict obliqueness estimate}

To establish the strict obliqueness estimate, firstly we require the following result.
\begin{lemma}\label{l3.1}
Let $\kappa_{1}(x)$, $\cdots$, $\kappa_{n}(x)$ be the principal curvatures of $\Gamma=\{(x,u(x))\mid x\in \Omega\}$ at $x$. Suppose that  $u\in  C^{\infty}(\bar{\Omega})$ is a uniformly convex solution of (\ref{e1.3})-(\ref{e1.4}), then there exists $M_{1}>0$ and $M_{2}>0$ depending only on
$\Omega$ and $\tilde{\Omega}$ such that
\begin{equation}\label{e3.1}
\min_{1\leq i\leq n}\kappa_i(x)\leq M_{1},\ \  \max_{1\leq i\leq n}\kappa_i(x)\geq M_{2}.
\end{equation}
\end{lemma}
\begin{proof}
From $Du(\Omega)=\tilde{\Omega}$, we obtain
$$\int_{\Omega}\det D^{2}u(x)dx=|\tilde{\Omega}|.$$
Denote $\lambda_{1}(x)$, $\cdots$, $\lambda_{n}(x)$ be the eigenvalues of $D^{2}u$ at $x$, then there exists $\bar x\in \bar{\Omega}$ such that
$$ \prod_{i=1}^n\lambda_i(\bar x)=\det D^{2}u(\bar x)=\frac{|\tilde{\Omega}|}{|\Omega|}=\theta_{0}.$$
By (\ref{e2.7}), we obtain
$$ \prod_{l=1}^n\kappa_l(\bar x)=\det\left(\frac{1}{v}b^{ik}b^{kj}\right)\det D^{2}u(\bar x)=\det\left(\frac{1}{v}b^{ik}b^{kj}\right)\prod_{l=1}^n\lambda_l(\bar x).$$
With the same method as the discussion (\ref{e2.15}), we have
$$\sigma_{1}\theta_{0}\leq \prod_{l=1}^n\kappa_l(\bar x)\leq\sigma_{2}\theta_{0}.$$
Using (\ref{e1.3}) and the monotonicity of $F$, we get
\begin{equation*}
\begin{aligned}
F\left(\min_{1\leq i\leq n}\kappa_i(x),\cdots, \min_{1\leq i\leq n}\kappa_i(x)\right)&\leq F\left(\kappa_1(x),\cdots,\kappa_n(x)\right)\\
&=F\left(\kappa_1(\bar x),\cdots,\kappa_n(\bar x)\right)\\
&\leq \arctan\left(\min_{1\leq i\leq n}\kappa_i(\bar x)\right)+\frac{(n-1)\pi}{2},\\
\end{aligned}
\end{equation*}
Therefore,
$$\arctan\left(\min_{1\leq i\leq n}\kappa_i( x)\right)\leq\frac{\arctan(\sigma_{2}\theta_{0})^{\frac{1}{n}}+\frac{(n-1)\pi}{2}}{n}<\frac{\pi}{2}$$
and
\begin{equation*}
\begin{aligned}
F\left(\max_{1\leq i\leq n}\kappa_i(x),\cdots, \max_{1\leq i\leq n}\kappa_i(x)\right)&\geq F\left(\kappa_1(x),\cdots,\kappa_n(x)\right)\\
&=F\left(\kappa_1(\bar x),\cdots,\kappa_n(\bar x)\right)\\
&\geq \arctan\left(\max_{1\leq i\leq n}\kappa_i(\bar x)\right)\\
&\geq \arctan(\sigma_{1}\theta_{0})^{\frac{1}{n}}.\\
\end{aligned}
\end{equation*}
Hence
$$n\arctan\left(\max_{1\leq i\leq n}\kappa_i(x)\right)\geq\arctan(\sigma_{1}\theta_{0})^{\frac{1}{n}}>0.$$
In conclusion,
\begin{equation*}
\min_{1\leq i\leq n}\kappa_i(x)\leq M_{1},\ \  \max_{1\leq i\leq n}\kappa_i(x)\geq M_{2},
\end{equation*}
where
\begin{equation*}
M_{1}=\tan\left(\frac{\arctan(\sigma_{2}\theta_{0})^{\frac{1}{n}}+\frac{(n-1)\pi}{2}}{n}\right),\qquad M_{2}=\tan\left(\frac{\arctan(\sigma_{1}\theta_{0})^{\frac{1}{n}}}{n}\right).
\end{equation*}
\end{proof}
By Lemma \ref{l3.1}, the points $(\kappa_{1},\cdots, \kappa_{n})$ are always in $ \Gamma^{+}_{]M_{1},M_{2}[}$ under the problem (\ref{e1.3})-(\ref{e1.4}). Then there exist $\Lambda_1>0$ and $\Lambda_2>0$ depending only on $\Omega$ and $\tilde{\Omega}$, such that $F$ satisfies the structure conditions Lemma \ref{l2.1} (iii). The Legendre transform of $u$ which is the convex function $u^{\ast}$ on $\tilde{\Omega}=Du(\Omega)$
satisfies the equation (\ref{e2.22}). The points $(\mu_{1},\cdots, \mu_{n})$  are also always in $ \Gamma^{+}_{]M_{1},M_{2}[}$
by $[a^{\ast}_{ij}]=[a_{ij}]^{-1}$. Due to the translational invariance of equation (\ref{e2.22}), it can be shown easily that $F^{\ast}=F$.
Then $F^{\ast}$ also satisfies the structure conditions Lemma \ref{l2.1} (iii).
In the following, we always assume that $\Lambda_1$ and $\Lambda_2$ are universal constants depending only on the known data.

Next, we will carry out the strictly oblique estimate. Let $\mathscr{P}_n$ be the set of positive definite symmetric $n\times n$ matrices, and $\kappa_{1}(A)$, $\cdots$, $\kappa_{n}(A)$ be the eigenvalues of $A$.  For $A=(a_{ij})\in \mathscr{P}_n$, recall
$$F[A]:=F\left(\kappa_{1}(A),\cdots, \kappa_{n}(A)\right)=\sum^{n}_{i=1}\arctan\kappa_{i}.$$
and
$$\left(a^{ij}\right)=(a_{ij})^{-1}.$$
Given the bound strictly convex  domain $\tilde{\Omega}$ with smooth boundary in $\mathbb{R}^{n}$,
there exists so-call boundary defining function as follows and the construction process can be seen in \cite{SM}.
\begin{deff}\label{d1}
A smooth function $h:\mathbb{R}^n\rightarrow\mathbb{R}$ is called the defining function of $\tilde{\Omega}$ if
$$\tilde{\Omega}=\{p\in\mathbb{R}^{n} : h(p)>0\},\quad |Dh|_{{\partial\tilde{\Omega}}}=1,$$
and there exists $\theta>0$ such that for any $p=(p_{1},\cdots, p_{n})\in \tilde{\Omega}$ and $\xi=(\xi_{1}, \cdots, \xi_{n})\in \mathbb{R}^{n}$,
$$\frac{\partial^{2}h}{\partial p_{i}\partial p_{j}}\xi_{i}\xi_{j}\leq -\theta|\xi|^{2}.$$
\end{deff}
Therefore, the diffeomorphism condition $Du(\Omega)=\tilde{\Omega}$ in (\ref{e1.4}) is equivalent to
\begin{equation}\label{e3.2}
 h(Du)=0,\ \ x\in \partial\Omega.
\end{equation}
Then (\ref{e1.3})-(\ref{e1.4}) can be rewritten as
\begin{equation}\label{e3.3}
\left\{ \begin{aligned}
F\left[A\right]&=c,\ \ &&x\in \Omega, \\
h(Du)&=0,             &&x\in\partial \Omega,
\end{aligned} \right.
\end{equation}
where $A$ is denoted by (\ref{e2.8}).

This is an oblique boundary value problem of second order fully nonlinear elliptic equation. We also denote $\beta=(\beta^{1}, \cdots, \beta^{n})$ with $\beta^{i}:=h_{p_{i}}(Du)$, and $\nu=(\nu_{1},\cdots,\nu_{n})$ as the unit inward normal vector at $x\in\partial\Omega$.
The expression of the inner product  is
\begin{equation*}
\langle\beta, \nu\rangle=\beta^{i}\nu_{i}.
\end{equation*}

\begin{lemma}\label{l3.3}(See J. Urbas \cite{JU}.) Let $\nu=(\nu_{1},\nu_{2}, \cdots,\nu_{n})$ be the unit inward normal vector of $\partial\Omega$. If $u\in C^{2}(\bar{\Omega})$  with $D^{2}u\geq0$,
then there holds $h_{p_{k}}(Du)\nu_{k}\geq0$.
\end{lemma}
\begin{lemma}\label{l3.4} Assume that  $[A_{ij}]$ is semi-positive real symmetric matrix and $[B_{ij}]$, $[C_{ij}]$ are two real symmetric matrixes. Then
$$2A_{ij}B_{jk}C_{ki}\leq A_{ij}B_{ik}B_{jk}+A_{ij}C_{ik}C_{jk}.$$
\end{lemma}
\begin{proof}
We see that $A(B-C)(B-C)$ is a semi-positive real symmetric matrix. Then
$$\mathrm{Tr}(A(B-C)(B-C))\geq 0.$$
Then we obtain
$$\mathrm{Tr}(ABC+ACB)\leq \mathrm{Tr}(ABB+ACC).$$
By direct computations, one can deduce that
$$A_{ij}B_{jk}C_{ki}+A_{ij}C_{jk}B_{ki}\leq A_{ij}B_{ik}B_{jk}+A_{ij}C_{ik}C_{jk},$$
by making use of $B=B^{T}, C=C^{T}$. It's easy to see that $A_{ij}B_{jk}C_{ki}=A_{ij}C_{jk}B_{ki}$, then we obtain the desired results.
\end{proof}
Now, we can present
\begin{lemma}\label{l3.5}
 If $u$ is a uniformly smooth convex solution of (\ref{e3.3}), then the strict obliqueness estimate
\begin{equation}\label{e3.4}
\langle\beta, \nu\rangle\geq \frac{1}{C_{1}}>0,
\end{equation}
holds on $\partial \Omega$ for some universal constant $C_{1}$, which depends only on $\Omega$ and $\tilde{\Omega}$.
\end{lemma}
\begin{rem}
Without loss of generality, in the following we set $C_{1},C_{2}, \cdots,$ to be constants depending only on $\Omega$ and $\tilde{\Omega}$.
\end{rem}
\begin{proof}It is a modification of the proof of Lemma 3.5 in \cite{CHB}.
Define
$$\omega=\langle \beta,\nu\rangle+\tau h(Du),$$
where $\tau$ is a positive constant to be determined. Let $x_0\in \partial \Omega$ such that
$$\langle \beta,\nu\rangle(x_0)=h_{p_k}(Du(x_0))\nu_k(x_0)=\min_{\partial\Omega}\langle \beta,\nu\rangle.$$
By rotation, we may assume that $\nu(x_0)=(0,\cdots,0,1)$. Applying the above assumptions and the boundary condition, we find that
$$\omega(x_0)=\min_{\partial\Omega} \omega=h_{p_{n}}(Du(x_{0})).$$
By the smoothness of $\Omega$ and its convexity, we extend $\nu$ smoothly to a tubular neighborhood of $\partial\Omega$ such that in the matrix sense
\begin{equation}\label{e3.5}
  \left(\nu_{kl}\right):=\left(D_k\nu_l\right)\leq -\frac{1}{C}\operatorname{diag} (\underbrace {1,\cdots, 1}_{n-1},0),
\end{equation}
where $C$ is a positive constant. By Lemma \ref{l3.3}, we see that $h_{p_{n}}(Du(x_{0}))\geq0$.

At the maximum point  $x_0$  one may show that
\begin{equation}\label{e3.6}
 0=\omega_r=h_{p_np_k}u_{kr}+h_{p_k}\nu_{kr}+\tau h_{p_k}u_{kr}, \quad 1\leq r\leq n-1.
\end{equation}
We assume that the following key result
\begin{equation}\label{e3.7}
 \omega_n(x_0)>-C,
\end{equation}
holds which will be proved later, where $C$ is a positive constant depending only on $\Omega$, and $h^{\ast}$. We observe that (\ref{e3.7}) can be rewritten as
\begin{equation}\label{e3.8}
h_{p_np_k}u_{kn}+h_{p_k}\nu_{kn}+ \tau h_{p_k}u_{kn}>-C.
\end{equation}
Multiplying $h_{p_n}$ on both sides of (\ref{e3.8}) and $h_{p_r}$ on both sides of (\ref{e3.6}) respectively, and summing up together, one gets
\begin{equation*}
 \tau h_{p_k}h_{p_l}u_{kl}\geq -Ch_{p_n}- h_{p_k}h_{p_l}\nu_{kl}- h_{p_k}h_{p_np_l}u_{kl}.
\end{equation*}
Combining (\ref{e3.5}) and
$$ 1\leq r\leq n-1,\quad h_{p_k}u_{kr}=\frac{\partial h(Du)}{\partial x_r}=0,\quad h_{p_k}u_{kn}=\frac{\partial h(Du)}{\partial x_n}\geq 0,\quad -h_{p_np_n}\geq 0,$$
we have
$$\tau h_{p_k}h_{p_l}u_{kl}\geq-Ch_{p_n}+\frac{1}{C}|Dh|^2-\frac{1}{C}h^2_{p_n}=-Ch_{p_n}+\frac{1}{C}-\frac{1}{C}h^2_{p_n}.$$
Now, to obtain the estimate $\langle \beta,\nu\rangle$ we divide $-Ch_{p_n}+\frac{1}{C}-\frac{1}{C}h^2_{p_n}$ into two cases at $x_0$.

Case (i).  If
$$-Ch_{p_n}+\frac{1}{C}-\frac{1}{C}h^2_{p_n}\leq \frac{1}{2C},$$
then
$$h_{p_k}(Du)\nu_{k}=h_{p_n }\geq \sqrt{\frac{1}{2}+\frac{C^4}{4}}-\frac{C^2}{2}.$$
It means that there is a uniform non-negative lower bound for $\underset{\partial\Omega}\min \langle \beta,\nu\rangle$.

Case (ii). If
$$-Ch_{p_n}+\frac{1}{C}-\frac{1}{C}h^2_{p_n}> \frac{1}{2C},$$
then we know that there is a non-negative lower bound for $h_{p_k}h_{p_l}u_{kl}$.

Let $u^{\ast}$ be the Legendre transformation of $u$, then by (\ref{e2.21aa}) $u^{\ast}$ satisfies
\begin{equation}\label{e3.10}
\left\{ \begin{aligned}G^{\ast}(y,D^2u^{\ast}) &=\frac{n\pi}{2}-c, \ \ && y\in\tilde{\Omega},\\
h^{\ast}(Du^{\ast})&=0, && y\in\partial\tilde{\Omega},
\end{aligned} \right.
\end{equation}
where $h^{\ast}$ is the defining function of $\Omega$. That is,
$$\Omega=\{\tilde{p}\in\mathbb{R}^{n} : h^{\ast}(\tilde{p})>0\},\ \ \ |Dh^{\ast}|_{\partial\Omega}=1, \ \ \ D^2h^{\ast}\leq -\tilde{\theta}I,$$
where $\tilde{\theta}$ is a positive constant. Here we emphasize that $G^{\ast}=G$ by (\ref{e2.20}) and (\ref{e2.22}).
The unit inward normal vector of $\partial\Omega$ can be represented as $\nu=Dh^{\ast}$. By the same token,
$\tilde{\nu}=Dh$, where $\tilde{\nu}=(\tilde{\nu}_{1}, \tilde{\nu}_{2},\cdots,\tilde{\nu}_{n})$ is the unit inward normal vector of $\partial\tilde{\Omega}$.

Let $\tilde{\beta}=(\tilde{\beta}^{1}, \cdots, \tilde{\beta}^{n})$ with $\tilde{\beta}^{k}:=h^{\ast}_{p_{k}}(Du^{\ast})$.
Using the representation as the works of \cite{HO},\cite{OK},\cite{HRY},
we also define
$$\tilde{\omega}=\langle\tilde{\beta}, \tilde{\nu}\rangle+\tilde{\tau} h^{\ast}(Du^{\ast}),$$
in which
$$\langle\tilde{\beta}, \tilde{\nu}\rangle=\langle\beta, \nu\rangle,$$
and $\tilde{\tau}$ is a  positive constant to be determined.
Denote $y_{0}=Du(x_{0})$. Then $$\tilde{\omega}(y_{0})=\omega(x_{0})=\min_{\partial\tilde{\Omega}} \tilde{\omega}.$$ Using the same methods, under the assumption of
\begin{equation}\label{e3.11}
\tilde{\omega}_{n}(y_{0})\geq -C,
\end{equation}
we obtain the non-negative lower bounds of $h^{\ast}_{p_{k}}h^{\ast}_{p_{l}}u^{\ast}_{kl}$ or
$$h_{p_{k}}(Du)\nu_{k}=h^{\ast}_{p_{k}}(Du^{\ast})\tilde{\nu}_{k}=h^{\ast}_{p_{n}}\geq\sqrt{\frac{1}{2}+\frac{C^4}{4}}-\frac{C^2}{2}.$$
On the other hand, one can easily check that
$$h^{\ast}_{p_{k}}h^{\ast}_{p_{l}}u^{\ast}_{kl}=\nu_{i}\nu_{j}u^{ij}.$$
Then by the non-negative lower bounds of $h_{p_{k}}h_{p_{l}}u_{kl}$ and $h^{\ast}_{p_{k}}h^{\ast}_{p_{l}}u^{\ast}_{kl}$, the desired conclusion can be obtained by
\begin{equation*}
\langle \beta,\nu\rangle=\sqrt{h_{p_k}h_{p_l}u_{kl}u^{ij}\nu_i\nu_j}.
\end{equation*}
It remains to prove the key estimates (\ref{e3.7}) and (\ref{e3.11}).

 At first we give the proof of (\ref{e3.7}). By (\ref{e2.11}), Lemma \ref{l3.4} and the boundedness of $G_{i}$, we have
\begin{equation}\label{e3.13}
 \begin{aligned}
\mathcal{L}\omega=&G_{ij}u_{il}u_{jm}(h_{p_{k}p_{l}p_{m}}\nu_{k}+\tau h_{p_{l}p_{m}})\\
&+2G_{ij}h_{p_{k}p_{l}}u_{li}\nu_{kj}+G_{ij}h_{p_{k}}\nu_{kij}+G_{i}h_{p_{k}}\nu_{ki}\\
\leq& (h_{p_{k}p_{l}p_{m}}\nu_{k}+\tau h_{p_{l}p_{m}}+\delta_{lm})G_{ij}u_{il}u_{jm}+C_{2}\mathcal{T}_{G}+C_{3},
 \end{aligned}
\end{equation}
where $\mathcal{L}=G_{ij}\partial_{ij}+G_{i}\partial_{i}$ and
$$2G_{ij}h_{p_{k}p_{l}}u_{li}\nu_{kj}\leq  G_{ij}u_{im}u_{mj}+C_{2}\mathcal{T}_{G}.$$
Since $D^{2}h\leq-\theta I$, we may choose $\tau$ large enough depending on the known data such that
$$(h_{p_{k}p_{l}p_{m}}\nu_{k}+\tau h_{p_{l}p_{m}}+\delta_{lm})<0.$$
Consequently, we deduce that
\begin{equation}\label{e3.14}
\mathcal{L}\omega\leq C_{4}\mathcal{T}_{G} \    \ in \   \ \Omega.
\end{equation}
by the convexity of $u$.

Denote a neighborhood of $x_{0}$ in $\Omega$ by
$$\Omega_{r}:=\Omega\cap B_{r}(x_{0}),$$
where $r$ is a positive constant such that $\nu$ is well defined in $\Omega_{r}$. In order to obtain the desired results, it suffices to consider the auxiliary function
$$\Phi(x)=\omega(x)-\omega(x_{0})+\sigma\ln(1+kh^{\ast}(x))+A|x-x_{0}|^{2},$$
where $\sigma$, $k$ and $A$ are positive constants to be determined.
By  $h^{\ast}$ being the defining function of $\Omega$ and $G_{i}$ being bounded one can show that
\begin{equation}\label{e3.15}
 \begin{aligned}
  \mathcal{L}(\ln(1+kh^{\ast}))&=G_{ij}\left(\frac{kh^{\ast}_{ij}}{1+kh^{\ast}}
  -\frac{kh^{\ast}_{i}}{1+kh^{\ast}}\frac{kh^{\ast}_{j}}{1+kh^{\ast}}\right)
  +G_{i}\frac{kh^{\ast}_{i}}{1+kh^{\ast}}\\
  &\triangleq G_{ij}\frac{kh^{\ast}_{ij}}{1+kh^{\ast}}-G_{ij}\eta_{i}\eta_{j}+G_{i}\eta_{i}\\
  &\leq\left(-\frac{k\tilde{\theta}}{1+kh^{\ast}}+C_{5}-C_{6}|\eta-C_{7}I|^{2}\right)\mathcal{T}_{G}\\
  &\leq\left(-\frac{k\tilde{\theta}}{1+kh^{\ast}}+C_{5}\right)\mathcal{T}_{G},
  \end{aligned}
\end{equation}
where $\eta=\left(\frac{kh^{\ast}_{1}}{1+kh^{\ast}}, \frac{kh^{\ast}_{2}}{1+kh^{\ast}},\cdots, \frac{kh^{\ast}_{n}}{1+kh^{\ast}}\right)$.

By taking $r$ to be small enough such that we have
\begin{equation}\label{e3.16}
 \begin{aligned}
0\leq h^{\ast}(x)&=h^{\ast}(x)-h^{\ast}(x_{0})\\
&\leq \sup_{\Omega_{r}}|Dh^{\ast}||x-x_{0}|\\
&\leq r\sup_{\Omega}|Dh^{\ast}|\leq \frac{\tilde{\theta}}{3C_{5}}.
 \end{aligned}
\end{equation}

By choosing $k=\frac{7C_{5}}{\tilde{\theta}}$ and applying (\ref{e3.16}) to (\ref{e3.15}) we obtain
\begin{equation}\label{e3.17}
  \mathcal{L}(\ln(1+kh^{\ast}))
  \leq -C_{5}\mathcal{T}_{G}.
\end{equation}
Combining (\ref{e3.14}) with (\ref{e3.17}), a direct computation yields
\begin{equation*}
\mathcal{L}(\Phi(x))\leq (C_{4}-\sigma C_{5}+2A+2AC_{8})\mathcal{T}_{G}.
\end{equation*}
On $\partial\Omega$, it is clear that $\Phi(x)\geq0$.
Because $\omega$ is bounded, then it follows that we can choose $A$ large enough depending on the known data such that on $\Omega\cap\partial B_{r}(x_{0})$,
\begin{equation*}
 \begin{aligned}
  \Phi(x)&=\omega(x)-\omega(x_{0})+\sigma\ln(1+kh^{\ast})+Ar^{2}\\
  &\geq\omega(x)-\omega(x_{0})+Ar^{2}\geq 0.
  \end{aligned}
\end{equation*}
Let
$$\sigma=\frac{C_{4}+2A+2AC_{8}}{C_{5}},$$
Consequently,
\begin{equation}\label{e3.18}
\left\{ \begin{aligned}
   \mathcal{L}\Phi&\leq 0,\ \  &&x\in\Omega_{r},\\
   \Phi&\geq 0,\ \  &&x\in\partial\Omega_{r}.
                          \end{aligned} \right.
\end{equation}
According to the maximum principle, it follows that
$$\Phi|_{\Omega_{r}}\geq \min_{\partial\Omega_{r}}\Phi\geq 0.$$
By the above inequality and $\Phi(x_0)=0$, we have $\partial_n\Phi(x_0)\geq 0$, which gives the desired estimate (\ref{e3.7}).

Finally, we are turning to the proof of (\ref{e3.11}). The proof of (\ref{e3.11}) is similar to that of (\ref{e3.7}). Define
$$\mathcal{\tilde{L}}=G^{\ast}_{ij}\partial_{ij}.$$
From (\ref{e2.27}) and $G^{\ast}_{y_{k}}$ being bounded, we get
\begin{equation*}
 \begin{aligned}
\mathcal{\tilde{L}}\tilde{\omega}=&G^{\ast}_{ij}u^{\ast}_{li}u^{\ast}_{mj}(h^{\ast}_{q_{k}q_{l}q_{m}}\tilde{\nu}_{k}+\tilde{\tau} h^{\ast}_{q_{l}q_{m}})+2G^{\ast}_{ij}h^{\ast}_{q_{k}q_{l}}u^{\ast}_{li}\tilde{\nu}_{kj}\\
&-G^{\ast}_{y_{k}}(h^{\ast}_{q_{k}q_{m}}\tilde{\nu}_{m}+\tilde{\tau} h^{\ast}_{q_{k}})+G^{\ast}_{ij}h^{\ast}_{q_{k}}\tilde{\nu}_{kij}\\
\leq&(h^{\ast}_{q_{k}q_{l}q_{m}}\tilde{\nu}_{k}+\tilde{\tau} h^{\ast}_{q_{l}q_{m}}+\delta_{lm})G^{\ast}_{ij}u^{\ast}_{il}u^{\ast}_{jm}+C_{9}\mathcal{T}^{\ast}_{G}+C_{10}(1+\tilde{\tau}),
 \end{aligned}
\end{equation*}
where
$$2G^{\ast}_{ij}h^{\ast}_{q_{k}q_{l}}u^{\ast}_{li}\tilde{\nu}_{kj}\leq \delta_{lm}G^{\ast}_{ij}u^{\ast}_{il}u^{\ast}_{jm}
+C_{9}\mathcal{T}^{\ast}_{G},$$
by Lemma \ref{l3.4}. Since $D^{2}h^{\ast}\leq-\tilde{\theta}I$, we only need to choose $\tilde{\tau}$ sufficiently large depending on the known data such that
$$h^{\ast}_{q_{k}q_{l}q_{m}}\tilde{\nu}_{k}+\tilde{\tau} h^{\ast}_{q_{l}q_{m}}+\delta_{lm}<0.$$
Therefore,
\begin{equation}\label{e3.19}
\mathcal{\tilde{L}}\tilde{\omega}\leq C_{11}\mathcal{T}^{\ast}_{G}.
\end{equation}
by the convexity of $u^{\ast}$.

Denote a neighborhood of $y_{0}$ in $\tilde{\Omega}$ by
$$\tilde{\Omega}_{\rho}:=\tilde{\Omega}\cap B_{\rho}(y_{0}),$$
where $\rho$ is a positive constant such that $\tilde{\nu}$ is well defined in $\tilde{\Omega}_{\rho}$. In order to obtain the desired results,
we  consider the  auxiliary function
$$\Psi(y)=\tilde{\omega}(y)-\tilde{\omega}(y_{0})+\tilde{k}h(y)+\tilde{A}|y-y_{0}|^{2},$$
where $\tilde{k}$ and $\tilde{A}$ are positive constants to be determined. It is easy to check that $\Psi(y)\geq0$ on $\partial\tilde{\Omega}$. Now that $\tilde{\omega}$ is bounded, it follows that we can choose $\tilde{A}$ large enough depending on the known data such that on $\tilde{\Omega}\cap\partial B_{\rho}(y_{0})$,
\begin{equation*}
  \Psi(y)=\tilde{\omega}(y)-\tilde{\omega}(y_{0})+\tilde{k}h(y)+\tilde{A}\rho^{2}\geq \tilde{\omega}(y)-\tilde{\omega}(y_{0})
  +\tilde{A}\rho^{2}\geq 0.
\end{equation*}
It follows from $D^{2}h\leq-\theta I$ that
$$\mathcal{\tilde{L}}(\tilde{k}h(y)+\tilde{A}|y-y_{0}|^{2})\leq(-\tilde{k}\theta+2\tilde{A})\mathcal{T}^{\ast}_{G}.$$
Then by (\ref{e3.19}) and choosing $\tilde{k}=\frac{2\tilde{A}+C_{11}}{\theta}$ we obtain
$$\mathcal{\tilde{L}}\Psi(y)\leq0.$$
Consequently,
\begin{equation*}
\left\{ \begin{aligned}
   \mathcal{\tilde{L}}\Psi&\leq 0,\ \  &&y\in\tilde{\Omega}_{\rho},\\
   \Psi&\geq 0,\ \  &&y\in\partial\tilde{\Omega}_{\rho}.
                          \end{aligned} \right.
\end{equation*}
The rest of the proof of (\ref{e3.11}) is the same as (\ref{e3.7}). Thus the proof of (\ref{e3.4}) is completed.
\end{proof}
\begin{rem}
The above detail proof involves that the estimate is independent on the upper bound of $\mathcal{T}_{G}$ and $\mathcal{T}_{G^{*}}$,i.e we need not the upper bound of
$\sum^{n}_{i=1}\frac{\partial F}{\partial \kappa_{i}}$  and the lower bound of $\sum^{n}_{i=1}\frac{\partial F}{\partial \kappa_{i}}\kappa^{2}_{i}$ in Lemma \ref{l2.1}.
\end{rem}

\section{The second derivative estimate}

Before deriving the global $C^2$ estimate, we first introduce a useful definition that provides a basic connection between (\ref{e1.3})-(\ref{e1.4}) and (\ref{e3.3}) and which will be useful for the sequel.
\begin{deff}\label{d1.10}
We say that  $u^{*}$ in (\ref{e3.10}) is a dual solution to (\ref{e3.3}).
\end{deff}
To carry out the global $C^2$ estimate, we use the following strategy that is to reduce the $C^2$ global estimate of $u$ and $u^{*}$ to the boundary. By differentiating the boundary condition $h(Du)=0$ in any tangential direction $\varsigma$, we have
\begin{equation}\label{e4.1}
   u_{\beta \varsigma}=h_{p_k}(Du)u_{k\varsigma}=0.
\end{equation}
However, the second order derivative of $u$ on the boundary is controlled by $u_{\beta \varsigma}$, $u_{\beta \beta}$ and $u_{\varsigma\varsigma}$. We now give the arguments as in \cite{JU} and one can see there for more details. At $x\in \partial\Omega$, any unit vector $\xi$ can be written in terms of a tangential component $\varsigma(\xi)$ and a component in the direction $\beta$ by
\begin{equation*}
\xi=\varsigma(\xi)+\frac{\langle \nu,\xi\rangle}{\langle\beta,\nu\rangle}\beta,
\end{equation*}
where  $\nu$ is an unit normal vector of $\partial\Omega$  and
$$\varsigma(\xi):=\xi-\langle \nu,\xi\rangle \nu-\frac{\langle \nu,\xi\rangle}{\langle\beta,\nu\rangle}\beta^T,\qquad
\beta^T:=\beta-\langle \beta,\nu\rangle \nu.$$
Since (\ref{e3.4}), we see that in fact
\begin{equation}\label{e4.3}
\begin{aligned}
|\varsigma(\xi)|^{2}&=1-\left(1-\frac{|\beta^{T}|^{2}}{\langle\beta,\nu\rangle^{2}}\right)\langle\nu,\xi\rangle^{2}
-2\langle\nu,\xi\rangle\frac{\langle\beta^{T},\xi\rangle}{\langle\beta,\nu\rangle}\\
&\leq 1+C\langle\nu,\xi\rangle^{2}-2\langle\nu,\xi\rangle\frac{\langle\beta^{T},\xi\rangle}{\langle\beta,\nu\rangle}\\
&\leq C.
\end{aligned}
\end{equation}
Denote $\varsigma:=\frac{\varsigma(\xi)}{|\varsigma(\xi)|}$, then we combine (\ref{e4.1}), (\ref{e4.3}) and (\ref{e3.4}) to obtain
\begin{equation}\label{e4.4}
\begin{aligned}
u_{\xi\xi}&=|\varsigma(\xi)|^{2}u_{\varsigma\varsigma}
+2|\varsigma(\xi)|\frac{\langle\nu,\xi\rangle}{\langle\beta,\nu\rangle}u_{\beta\varsigma}
+\frac{\langle\nu,\xi\rangle^{2}}{\langle\beta,\nu\rangle^{2}}u_{\beta\beta}\\
&=|\varsigma(\xi)|^{2}u_{\varsigma\varsigma}+\frac{\langle\nu,\xi\rangle^{2}}{\langle\beta,\nu\rangle^{2}}
u_{\beta\beta}\\
&\leq C(u_{\varsigma\varsigma}+u_{\beta\beta}),
\end{aligned}
\end{equation}
where $C$ depends only on $\Omega$, $\tilde \Omega$. Therefore, we only need to estimate $u_{\beta\beta}$ and $u_{\varsigma\varsigma}$ respectively.
\begin{lemma}\label{lem4.2}
If $u$ is a smooth uniformly convex solution of (\ref{e3.3}),  then there exists a positive constant $C_{12}$  such that
\begin{equation}\label{e4.5}
   0\leq \sup_{\partial\Omega}u_{\beta\beta}\leq C_{12}.
\end{equation}
\end{lemma}
\begin{proof}
Let $\tilde{h}=h(Du)$ and denote $\mathcal{L}=G_{ij}\partial_{ij}+G_{i}\partial_{i}$ as (\ref{e3.13}). We arrive at
\begin{equation*}
 \begin{aligned}
\mathcal{L}\tilde{h}&=G_{ij}\partial_{ij}\tilde{h}+G_{i}\partial_{i}\tilde{h}\\
&=h_{p_{k}}G_{ij}u_{ijk}+h_{p_{k}}G_{i}u_{ik}+G_{ij}u_{ik}u_{jl}h_{p_{k}p_{l}}\\
&=G_{ij}u_{ik}u_{jl}h_{p_{k}p_{l}}\\
&\geq -C,
 \end{aligned}
\end{equation*}
by (\ref{e2.11}), (\ref{e2.16}) and Lemma \ref{l2.1} (iii) for some positive constant $C$ depending only on the known data. One can define the auxiliary function as follows
\begin{equation*}
\varpi=\sigma\ln(1+\kappa h^{\ast}(x)).
\end{equation*}
By the computations in (\ref{e3.17}), there exist some constants $\sigma$ and $\kappa$ such that
$$\mathcal{L}\varpi\leq \mathcal{L}\tilde{h},\quad \mathrm{on}\quad \Omega.$$
It follows from $\varpi|_{\partial\Omega}=\tilde{h}_{\partial\Omega}=0$ and the maximum principle that
$$\varpi\geq\tilde{h},\quad \mathrm{on}\quad \Omega.$$
We observe that
$$\varpi|_{\partial\Omega}=\tilde{h}|_{\partial\Omega}=0.$$
 Then we arrive at
$$\varpi_{\beta}\geq \tilde{h}_{\beta}=u_{\beta\beta},\quad \mathrm{on}\quad \partial\Omega.$$
\end{proof}
To bound the remaining second derivatives one can carry out some computations on the graph $\Gamma$ as same as \cite{J} by making use of local orthonormal frame fields.

In a neighbourhood of any point of $\Gamma$,  there exists a local orthonormal frame field $\hat{e}_{1},\hat{e}_{2},\cdots,\hat{e}_{n}$ on  $\Gamma$.
We denote covariant differentiation on $\Gamma$ in the direction $\hat{e}_{i}$ by $\nabla_{i}$. Let
$$\mu=\frac{(-Du,1)}{\sqrt{1+|Du|^{2}}},$$
denotes the upwards pointing unit normal vector field to $\Gamma$ and $[h_{ij}]$ denotes the second fundamental form of $\Gamma$, so that for the frame $\hat{e}_{1},\hat{e}_{2},\cdots,\hat{e}_{n}$ we have
$$h_{ij}=\langle D_{\hat{e}_{i}}\hat{e}_{j},\mu\rangle,$$
where $D$ denotes the usual gradient operator on $\mathbb{R}^{n+1}$.

In order to control the global second  derivative by its value on the boundary
we will need the following differential inequalities for  the mean curvature.
\begin{lemma}\label{lem4.3}
If $u$ is a smooth uniformly convex solution of (\ref{e3.3}),  then we have
\begin{equation}\label{e4.6}
F_{ij}\nabla_{i}\nabla_{j}H\geq 0,
\end{equation}
where $H=\sum_{i=1}^{n}h_{ii}$ is the mean curvature of $\Gamma$.
\end{lemma}
\begin{proof}
First, we need to deal with  $F_{ij}\nabla_{i}\nabla_{j}h_{ll}$ in the similar way as in \cite{J2}. Applying $\Delta=\nabla_{l}\nabla_{l}$ to $F[A]=c$, we have
$$F_{ij,rs}\nabla_{l}h_{rs}\nabla_{l}h_{ij}+F_{ij}\nabla_{l}\nabla_{l}h_{ij}=0.$$
The next  procedure in the proof is to calculate  $\nabla_{l}\nabla_{l}h_{ij}$ by using the Codazzi equations which tell us that $\nabla_{l}h_{ij}$ is symmetric in all indices, together with the standard formula for handling covariant derivatives, and Gauss equations
$$R_{ijkl}=h_{ik}h_{jl}-h_{il}h_{jk},$$
where $R_{ijkl}$ denotes the curvature tensor on $\Gamma$. Direct computations yield
\begin{equation*}
 \begin{aligned}
\nabla_{l}\nabla_{l}h_{ij}&=\nabla_{l}\nabla_{i}h_{jl}\\
&=\nabla_{i}\nabla_{l}h_{jl}+R_{lijm}h_{ml}+R_{lilm}h_{mj}\\
&=\nabla_{i}\nabla_{j}h_{ll}+h_{lj}h_{im}h_{ml}-h_{lm}h_{ij}h_{ml}+h_{ll}h_{im}h_{mj}-h_{lm}h_{il}h_{mj}\\
&=\nabla_{i}\nabla_{j}h_{ll}+h_{ll}h_{im}h_{mj}-h_{lm}h_{ij}h_{ml},\\
 \end{aligned}
\end{equation*}
so that, we obtain
\begin{equation}\label{e4.7}
F_{ij}\nabla_{i}\nabla_{j}h_{ll}=-F_{ij,rs}\nabla_{l}h_{rs}\nabla_{l}h_{ij}-F_{ij}h_{im}h_{mj}h_{ll}+F_{ij}h_{ij}h_{lm}h_{lm}.
\end{equation}
 Thus  summing over $l=1,\ldots,n$ in (\ref{e4.7}) we have
\begin{equation}\label{e4.8}
F_{ij}\nabla_{i}\nabla_{j}H=-F_{ij,rs}\nabla_{l}h_{rs}\nabla_{l}h_{ij}-F_{ij}h_{im}h_{mj}H+F_{ij}h_{ij}h_{lm}h_{lm}.
\end{equation}

The first term on the right hand side of (\ref{e4.8}) is non-negative  due to the concavity of $F$. Therefore, one only need to consider
$$P:=F_{ij}h_{ij}h_{lm}h_{lm}-HF_{ij}h_{im}h_{mj},$$
being non-negative everywhere.

Without loss of generality, we may assume that $\hat{e}_{1},\ldots,\hat{e}_{n}$ have been chosen such that $[h_{ij}]$ is diagonal at the point at which we are computing with eigenvalues $\kappa_{1},\ldots,\kappa_{n}$. Hence we finally conclude that
\begin{equation*}
 \begin{aligned}
P&=\left(\sum\frac{\partial F}{\partial \kappa_{i}}\kappa_{i}\right)\left(\sum\kappa^{2}_{j}\right)-\left(\sum\kappa_{i}\right)\left(\sum\frac{\partial F}{\partial \kappa_{j}}\kappa^{2}_{j}\right)\\
&=\sum_{i,j}\left(\frac{\partial F}{\partial \kappa_{i}}\kappa_{i}\kappa^{2}_{j}-\frac{\partial F}{\partial \kappa_{j}}\kappa^{2}_{j}\kappa_{i}\right)\\
&=\frac{1}{2}\sum_{i,j}\left(\frac{\partial F}{\partial \kappa_{j}}-\frac{\partial F}{\partial \kappa_{i}})(\kappa_{i}-\kappa_{j}\right)\kappa_{i}\kappa_{j}\\
&=\frac{1}{2}\sum_{i,j}\left(\frac{1}{1+\kappa^{2}_{j}}-\frac{1}{1+\kappa^{2}_{i}}\right)\left(\kappa_{i}-\kappa_{j}\right)\kappa_{i}\kappa_{j}\\
&=\frac{1}{2}\sum_{i,j}\frac{1}{1+\kappa^{2}_{j}}\frac{1}{1+\kappa^{2}_{i}}\left(\kappa_{i}
+\kappa_{j}\right)\left(\kappa_{i}-\kappa_{j}\right)^{2}\kappa_{i}\kappa_{j}\\
&\geq 0,
 \end{aligned}
\end{equation*}
because the last estimate is positive by the convexity of $\Gamma$. The desired conclusion follow from the above.
\end{proof}

\begin{lemma}\label{lem4.4}
If $u$ is a smooth uniformly convex solution of (\ref{e3.3}).  Then  we have
\begin{equation}\label{e4.9}
\sup_{\Gamma}H\leq \sup_{\partial\Gamma}H.
\end{equation}
\end{lemma}
\begin{proof}
Using (\ref{e4.6}) and the maximum principle, the estimate (\ref{e4.9}) follows immediately.
\end{proof}

In order to control  the global second derivatives of $u$ on the boundary, we adopt the following notation. For any unit vector $\xi\in \mathbb{R}^{n}$, one can define
\begin{equation}\label{e4.10}
\tilde{\xi}=\frac{(\xi,D_{\xi}u)}{\sqrt{1+|D_{\xi}u|^{2}}},
\end{equation}
then $\tilde{\xi}$ is a unit tangent vector to $\Gamma$. So that (\ref{e4.10}) gives a one to one
correspondence from $\mathbb{R}^{n}$ to the tangent space $T_{x}\Gamma$ at $(x,u(x))$.
One can recall the  explicit expressions  of the normal curvature $k(\tilde{\xi})$ of $\Gamma$ in the direction $\tilde{\xi}$
as follows
\begin{equation}\label{e4.11}
k(\tilde{\xi})=\frac{D_{\xi\xi}u}{\sqrt{1+|Du|^{2}}(1+|D_{\xi}u|^{2})}.
\end{equation}
For the above details, we refer to the computations of \cite{J1}, Sect.2.
\begin{lemma}\label{lem4.5}
If $u$ is a smooth uniformly convex solution of (\ref{e3.3}).  Then  we have
\begin{equation}\label{e4.12}
\sup_{\Omega}|D^{2}u|\leq C_{13}\sup_{\partial\Omega}|D^{2}u|.
\end{equation}
\end{lemma}
\begin{proof}
According to the representation (\ref{e4.11}) of the normal curvature and the definition of the mean curvature, it follows the convexity of $\Gamma$ that
$$\frac{1}{C}D_{\xi\xi}u\leq H\leq C|D^{2}u|,$$
for any direction $\xi$ and some constant $C$ depending only on $\Omega$ and $\tilde{\Omega}$.
Thus from (\ref{e4.9}) we obtain
$$\sup_{\Omega}\frac{1}{C}D_{\xi\xi}u\leq\sup_{\Gamma} H\leq \sup_{\partial\Gamma}H\leq\sup_{\partial\Omega} C|D^{2}u|.$$
 This completes the proof of (\ref{e4.12}).
\end{proof}
\begin{lemma}\label{lem4.6}
If $u$ is a smooth uniformly convex solution of (\ref{e3.3}).  Then  we have
\begin{equation*}
\sup_{\Omega}|D^{2}u|\leq C_{14}\left(1+\sup_{x\in\partial\Omega,\varsigma\in T_{x}\partial\Omega, |\varsigma|=1}u_{\varsigma\varsigma}\right).
\end{equation*}
\end{lemma}
\begin{proof}
By virtue of (\ref{e4.4}) and (\ref{e4.5}) we see that for any direction $\xi$, on $\partial\Omega$,
\begin{equation*}
 \begin{aligned}
u_{\xi\xi}&=|\varsigma(\xi)|^{2}u_{\varsigma\varsigma}
+2|\varsigma(\xi)|\frac{\langle\nu,\xi\rangle}{\langle\beta,\nu\rangle}u_{\beta\varsigma}
+\frac{\langle\nu,\xi\rangle^{2}}{\langle\beta,\nu\rangle^{2}}u_{\beta\beta}\\
&\leq Cu_{\varsigma\varsigma}+C.
 \end{aligned}
\end{equation*}
Hence the desired result follows immediately by Lemma \ref{lem4.5}.
\end{proof}

The following equation will be useful later in this section.
\begin{lemma}\label{lem4.7}
If $u$ is a smooth uniformly convex solution of (\ref{e3.3}).  Then for any direction $\xi$ in $\mathbb{R}^{n+1}$, we have
\begin{equation}\label{e4.13}
F_{ij}\nabla_{i}\nabla_{j}\langle \mu,\xi\rangle+F_{ij}h_{ik}h_{jk}\langle \mu,\xi\rangle=0,
\end{equation}
where $$\mu=\frac{(-Du,1)}{\sqrt{1+|Du|^{2}}}.$$
\end{lemma}
\begin{proof}
By similar proof as \cite{J2}, we may assume that the vector fields $\hat{e}_{1},...,\hat{e}_{n}$ and $\mu$ have been extended in a $C^{2}$
fashion so they form an orthonormal frame  in a neighbourhood of the point in $\mathbb{R}^{n+1}$  at which we are computing.
Since $\mu=\frac{(-Du,1)}{\sqrt{1+|Du|^{2}}}$, a direct calculation yields
$$ \langle D_{\hat{e}_{j}}\mu,\mu\rangle=0,$$
and then we arrive at
\begin{equation}\label{e4.14}
 \begin{aligned}
D_{\hat{e}_{j}}\langle\mu,\xi\rangle=\langle D_{\hat{e}_{j}}\mu,\xi\rangle&=\langle D_{\hat{e}_{j}}\mu,\hat{e}_{k}\rangle\langle \hat{e}_{k},\xi\rangle+\langle D_{\hat{e}_{j}}\mu,\mu\rangle\langle \mu,\xi\rangle\\
&=-\langle\mu,D_{\hat{e}_{j}}\hat{e}_{k}\rangle\langle \hat{e}_{k},\xi\rangle\\
&=-h_{jk}\langle \hat{e}_{k},\xi\rangle,\\
 \end{aligned}
\end{equation}
by $h_{jk}=\langle D_{\hat{e}_{j}}\hat{e}_{k},\mu\rangle$. We now calculate $\nabla_{i}\nabla_{j}\langle \mu,\xi\rangle$. First, we recall the formula
$$\nabla_{j}f=D_{\hat{e}_{j}}f,\ \mathrm{for} \ f\in C^{\infty}$$
and
$$\nabla_{i}U_{j}=D_{\hat{e}_{i}}U_{j}-\Gamma^{k}_{ij}U_{k},\ \mathrm{for }\ U_{j}=\langle U,\hat{e}_{j}\rangle$$
where $U$ is a vector field and $\Gamma^{k}_{ij}$ are the coefficients of connection on $\Gamma$.  Then
$$\nabla_{i}\nabla_{j}\langle \mu,\xi\rangle=\nabla_{i}D_{\hat{e}_{j}}\langle \mu,\xi\rangle=D_{\hat{e}_{i}}D_{\hat{e}_{j}}\langle \mu,\xi\rangle-\Gamma^{k}_{ij}D_{\hat{e}_{k}}\langle \mu,\xi\rangle,$$
where $\Gamma^{k}_{ij}=\langle D_{\hat{e}_{i}}\hat{e}_{j},\hat{e}_{k}\rangle$. Hence,
\begin{equation}\label{e4.15}
 \begin{aligned}
  \nabla_{i}\nabla_{j}\langle \mu,\xi\rangle=&-D_{\hat{e}_{i}}[h_{jk}\langle \hat{e}_{k},\xi\rangle]+\Gamma^{k}_{ij}h_{kl}\langle \hat{e}_{l},\xi\rangle\\
=&-D_{\hat{e}_{i}}h_{jk}\langle \hat{e}_{k},\xi\rangle-h_{jk}\langle D_{\hat{e}_{i}}\hat{e}_{k},\mu\rangle\langle\mu,\xi\rangle\\
&-h_{jk}\langle D_{\hat{e}_{i}}\hat{e}_{k},\hat{e}_{l}\rangle\langle\hat{e}_{l},\xi\rangle+\Gamma^{k}_{ij}h_{kl}\langle \hat{e}_{l},\xi\rangle\\
=&-D_{\hat{e}_{i}}h_{jk}\langle \hat{e}_{k},\xi\rangle-\Gamma^{l}_{ik}h_{jk}\langle\hat{e}_{l},\xi\rangle
-h_{jk}h_{ik}\langle\mu,\xi\rangle+\Gamma^{k}_{ij}h_{kl}\langle \hat{e}_{l},\xi\rangle\\
=&-D_{\hat{e}_{i}}h_{jl}\langle \hat{e}_{l},\xi\rangle+\Gamma^{k}_{il}h_{jk}\langle\hat{e}_{l},\xi\rangle+\Gamma^{k}_{ij}h_{kl}\langle \hat{e}_{l},\xi\rangle-h_{jk}h_{ik}\langle\mu,\xi\rangle\\
=&-\nabla_{i}h_{jl}\langle \hat{e}_{l},\xi\rangle-h_{jk}h_{ik}\langle\mu,\xi\rangle\\
=&-\nabla_{l}h_{ji}\langle \hat{e}_{l},\xi\rangle-h_{jk}h_{ik}\langle\mu,\xi\rangle.
 \end{aligned}
\end{equation}
In the last four lines of the above, we used the fact that $\Gamma^{l}_{ik}=-\Gamma^{k}_{il}$, the standard formula
$$\nabla_{i}h_{jl}=D_{\hat{e}_{i}}h_{jl}-\Gamma^{k}_{il}h_{jk}-\Gamma^{k}_{ij}h_{kl},$$
and the Codazzi equations. Furthermore, since $F[A]=c$, then $F_{ij}\nabla_{l}h_{ij}=0,$
i.e. $F_{ij}\nabla_{l}h_{ij}\langle\hat{e}_{l},\xi\rangle=0$. By substituting (\ref{e4.15}) into the equation we obtain the desired result.
\end{proof}

The following pre-knowledge was learned from \cite{J} for calculating the geometric quantity which satisfies the useful differential inequality on manifold.  By (\ref{e4.11}) we can make an assumption of
\begin{equation*}
\frac{k(\tilde{\xi})}{\mu_{n+1}}=\sup_{x\in\partial\Omega,\varsigma\in T_{x}\partial\Omega, |\varsigma|=1}\frac{D_{\varsigma\varsigma}u}{(1+|D_{\varsigma}u|^{2})}.
\end{equation*}
restricted to directions $\varsigma$  which are tangential to $\partial\Omega$  at $x$. Without loss of generality, we may take $x$ to be the origin
and it is often convenient to choose coordinate systems such that $\xi=e_{1}\triangleq (1,0,\cdots,0).$ Then we have
\begin{equation*}
\tilde{e}_{1}=\frac{(e_{1},D_{1}u)}{\sqrt{1+|D_{1}u|^{2}}}.
\end{equation*}
Let $\varsigma$ be a smooth unit tangent vector field on $\partial\Omega\cap B_{\rho}(0)$ for some $\rho>0$ such that
$\varsigma(0)=\tilde{e}_{1}$ and $\varsigma$ can been smoothly extended to $\bar{\Omega}\cap \overline{B_{\rho}(0)}$
with $|\varsigma(x)|\equiv1$ for any $x\in\bar{\Omega}\cap \overline{B_{\rho}(0)}$ where $\rho$ depends only on $\Omega$.

One can lift $\varsigma$ a vector field on $\Gamma$  by (\ref{e4.11}) and then we obtain a unit tangent
vector field  $\tilde{\varsigma}$ on $\Gamma_{\rho}\triangleq\Gamma\cap (\overline{{B}_{\rho}(0)}\times \mathbb{R})$ such that $\tilde{\varsigma}$  is
is tangential to $\partial\Gamma$ on $\Gamma_{\rho}\cap\partial\Gamma$.

By making use of Gram-Schmidt orthogonalization to the basis $\{\tilde{\varsigma}, \tilde{e}_{2},\cdots,\tilde{e}_{n}\}$
where
\begin{equation*}
\tilde{e}_{j}=\frac{(e_{j},D_{j}u)}{\sqrt{1+|D_{j}u|^{2}}} \,\qquad\,(j=2,3,\cdots,n)),
\end{equation*}
we get  a local orthonormal frame field on $\Gamma_{\rho}$ denoted by $\{\hat{e}_{1}\triangleq\tilde{\varsigma}, \hat{e}_{2}, \cdots,\hat{e}_{n}\}$.

Under the local orthonormal frame field $\{\hat{e}_{1}, \hat{e}_{2}, \cdots,\hat{e}_{n}\}$, it induces the second fundamental form ($h_{ij}$)
according to $\Gamma$. Thus
\begin{equation}\label{e4.17}
k(\hat{e}_{1})=h_{11}.
\end{equation}
With the aid of  the representation  in (\ref{e4.11}), we get
\begin{equation}\label{e4.18}
\frac{h_{11}}{\mu_{n+1}}|_{\Gamma_{\rho}\cap\partial\Gamma}\leq \frac{h_{11}(0)}{\mu_{n+1}(0)}.
\end{equation}
Considering $W\triangleq\frac{h_{11}}{\mu_{n+1}}$, we can deduce that $W$ satisfies the following graceful differential inequality
on $\Gamma_{\rho}$ in order to estimate $h_{11}$ on the boundary.
\begin{lemma}\label{lem4.8}
If $u$ is a smooth uniformly convex solution of (\ref{e3.3}).  Then  we have
\begin{equation*}
F_{ij}\nabla_{i}\nabla_{j}W+b_{i}\nabla_{i}W\geq 0, \qquad on \quad\Gamma_{\rho},
\end{equation*}
where $b=(b_{1},b_{2},\cdots,b_{n})$ is a bounded vector field.
\end{lemma}
\begin{proof}
We directly compute
$$\nabla_{j}W=\frac{\mu_{n+1}\nabla_{j}h_{11}-h_{11}\nabla_{j}\mu_{n+1}}{\mu^{2}_{n+1}},$$
and
\begin{equation*}
 \begin{aligned}
 \nabla_{i}\nabla_{j}W=&\frac{\nabla_{i}\nabla_{j}h_{11}}{\mu_{n+1}}-\frac{h_{11}}{\mu^{2}_{n+1}}\nabla_{i}\nabla_{j}\mu_{n+1}\\
 &-\frac{\nabla_{j}h_{11}\nabla_{i}\mu_{n+1}}{\mu^{2}_{n+1}}-\frac{\nabla_{i}h_{11}\nabla_{j}\mu_{n+1}}{\mu^{2}_{n+1}}
 +\frac{2h_{11}}{\mu^{3}_{n+1}}\nabla_{i}\mu_{n+1}\nabla_{j}\mu_{n+1}\\
=&\frac{\nabla_{i}\nabla_{j}h_{11}}{\mu_{n+1}}-\frac{h_{11}}{\mu^{2}_{n+1}}\nabla_{i}\nabla_{j}\mu_{n+1}
-\frac{\nabla_{i}\mu_{n+1}}{\mu_{n+1}}\nabla_{j}W-\frac{\nabla_{j}\mu_{n+1}}{\mu_{n+1}}\nabla_{i}W.
\end{aligned}
 \end{equation*}
Furthermore, from equation (\ref{e4.13}), we have
\begin{equation}\label{e4.19}
F_{ij}\nabla_{i}\nabla_{j}\mu_{n+1}+F_{ij}h_{ik}h_{jk}\mu_{n+1}=0.
\end{equation}
Therefore, we have
\begin{equation*}
 \begin{aligned}
F_{ij}\nabla_{i}\nabla_{j}W&+F_{ij}\left(\frac{\nabla_{i}\mu_{n+1}}{\mu_{n+1}}\nabla_{j}W
+\frac{\nabla_{j}\mu_{n+1}}{\mu_{n+1}}\nabla_{i}W\right)\\
=&F_{ij}\frac{\nabla_{i}\nabla_{j}h_{11}}{\mu_{n+1}}-F_{ij}\frac{h_{11}}{\mu^{2}_{n+1}}\nabla_{i}\nabla_{j}\mu_{n+1}\\
=&-\frac{1}{\mu_{n+1}}F_{ij,rs}\nabla_{1}h_{rs}\nabla_{1}h_{ij}+\frac{h_{11}}{\mu^{2}_{n+1}}F_{ij}h_{ik}h_{jk}\mu_{n+1}\\
&-\frac{1}{\mu_{n+1}}\left(F_{ij}h_{im}h_{mj}h_{11}-F_{ij}h_{ij}h_{1m}h_{1m}\right)\\
\geq & \frac{1}{\mu_{n+1}}F_{ij}h_{ij}h_{1m}h_{1m}\geq 0,\\
\end{aligned}
 \end{equation*}
by (\ref{e4.7}), (\ref{e4.19}) and the concavity of $F$.
Next, by (\ref{e4.14}) we have
$$b_{i}=2F_{ij}\frac{\nabla_{j}\mu_{n+1}}{\mu_{n+1}}=-2F_{ij}h_{jk}\frac{\langle\hat{e}_{k},e_{n+1}\rangle}{\mu_{n+1}}.$$
 The convexity of $\Gamma$ and (\ref{e2.17}) imply that $[F_{ij}h_{jk}]$ is bounded. Therefore, $b_{i}$ is bounded. The proof of the lemma is finished.
\end{proof}
\begin{lemma}\label{lem4.9}
If $u$ is a smooth uniformly convex solution of (\ref{e1.3}) and (\ref{e1.4}).  Then  we get
\begin{equation*}
0\leq\sup_{x\in\partial\Omega,\varsigma\in T_{x}\partial\Omega, |\varsigma|=1}u_{\varsigma\varsigma}\leq C_{15}.
\end{equation*}
\end{lemma}
\begin{proof}
By the previous assumption, it is enough that we estimate $u_{11}(0)$.
Without loss of generality, we assume that $W(0)\geq 1$. One can set $$Z=\frac{W}{W(0)},$$ and define
$$\mathcal{L}=F_{ij}\nabla_{i}\nabla_{j}+b_{i}\nabla_{i}.$$
By Lemma \ref{lem4.6}, \ref{lem4.8} and using (\ref{e4.18}), we have
\begin{equation}\label{e4.20}
\begin{aligned}
\mathcal{L}Z&\geq -C_{16}\quad  on \quad \Gamma_{\rho}, \\
|Z|&\leq C_{17}\qquad on \quad \partial\Gamma_{\rho}, \\
Z&\leq 1\qquad\quad on \quad \Gamma_{\rho}\cap\partial\Gamma, \\
Z(0)&= 1.
\end{aligned}
\end{equation}
We consider the barrier function
$$\Phi(X)=1+\sigma\ln(1+kh^{\ast}(x))+A|x|^{2}$$
for $X=(x,u(x))\in \Gamma.$
As same as  the details in the  proof of (\ref{e3.18}), there exist constants $\sigma$, $k$, $A$ depending on the known data such that we get
\begin{equation}\label{e4.21}
\begin{aligned}
\mathcal{L}\Phi&\leq -C_{16}\quad on \quad \Gamma_{\rho}, \\
\Phi&\geq C_{17}\qquad on \quad \partial\Gamma_{\rho}, \\
\Phi&\geq 1\qquad\quad on \quad \Gamma_{\rho}\cap\partial\Gamma, \\
\Phi(0)&= 1.
\end{aligned}
\end{equation}
Combining (\ref{e4.20})  and (\ref{e4.21}), the maximum principle implies that
\begin{equation*}
\Phi-\Phi(0)\geq Z-Z(0)\quad on \quad \Gamma_{\rho}.
\end{equation*}
So that we obtain
\begin{equation*}
D_{\beta}Z(0)\leq D_{\beta}\Phi(0)\leq C_{18}.
\end{equation*}
Then
\begin{equation}\label{e4.22}
D_{\beta}W(0)\leq C_{18}W(0).
\end{equation}
Recalling (\ref{e4.11})  and (\ref{e4.17}) we obtain
\begin{equation*}
W=\frac{D_{\varsigma\varsigma}u}{1+|D_{\varsigma}u|^{2}}
\end{equation*}
for a smooth tangent vector field $\varsigma$ on $\partial\Omega$ with $\varsigma(0)=e_{1}$.  From (\ref{e4.22})
a direct computation shows that
\begin{equation*}
\frac{D_{11\beta}u}{1+|D_{1}u|^{2}}-2\frac{D_{11}u}{(1+|D_{1}u|^{2})^{2}}D_{1}uD_{1\beta}u\leq C_{18} D_{11}u.
\end{equation*}
Thus, by (\ref{e4.1}) and the second boundary condition, we obtain
\begin{equation}\label{e4.23}
D_{11\beta}u\leq C_{19} D_{11}u, \quad \mathrm{at}\quad X=(0,u(0))=(0,0).
\end{equation}
On the other hand, differentiating $h(Du)=0$ twice in the direction $e_1$ at $0$, we have
$$h_{p_k}D_{11k}u+h_{p_kp_l}D_{k1}uD_{l1}u=0.$$
Let $M=D_{11}u(0)$. From Definition \ref{d1}  the concavity of $h$ yields that
$$h_{p_k}D_{11k}u=-h_{p_kp_l}D_{k1}uD_{l1}u\geq \theta M^2.$$
Combining it with $h_{p_k}D_{11k}u=D_{11\beta}u$, and using (\ref{e4.23}) we obtain
$$\theta M^2\leq C_{18}M.$$
Then we get the upper bound of $M=D_{11}u(0)$ and thus the desired result follows.
\end{proof}
In terms of Lemma \ref{lem4.6}, \ref{lem4.9}, we see that
\begin{lemma}\label{lem4.10}
If $u$ is a smooth uniformly convex solution of (\ref{e1.3}) and (\ref{e1.4}). Then
\begin{equation}\label{e4.24}
\max_{\bar{\Omega}}|D^2u| \leq C_{20}.
\end{equation}
\end{lemma}
In the following,  we describe the positive lower bound of $D^{2}u$. For (\ref{e1.3}), in consider of the Legendre transformation of $u$, the function $u^{*}$ satisfies (\ref{e2.20}) and (\ref{e2.21})
where the structure condition $G^{\ast}$ as same as $G$ according to (\ref{e2.22}).
 Repeating the proof of Lemma \ref{lem4.10}, we have
\begin{lemma}\label{lem4.11}
If $u$ is a smooth uniformly convex solution of (\ref{e1.3}) and (\ref{e1.4}). Then the function $u^{*}$ satisfies
\begin{equation*}
\max_{\bar{\tilde{\Omega}}}|D^2u^{*}| \leq C_{21}.
\end{equation*}
\end{lemma}
By Lemma \ref{lem4.10} and Lemma \ref{lem4.11}, we conclude that
\begin{lemma}\label{lem4.12}
If $u$ is a smooth uniformly convex solution of (\ref{e1.3}) and (\ref{e1.4}). Then
\begin{equation*}
\frac{1}{ C_{22}}I\leq D^2 u(x) \leq  C_{22} I,\ \ x\in\bar\Omega,
\end{equation*}
where $I$ is the $n\times n$ identity matrix.
\end{lemma}

\section{Proof of Theorem \ref{t1.1}}

The goal of this section is to prove Theorem \ref{t1.1}. There are two key lemmas when we start the proof. Though the discussion on the uniqueness in the fifth section of \cite{SM}, we conclude that
\begin{lemma}\label{yinli2.4.1}
If  $u\in C^{\infty}(\bar{\Omega})$ are uniformly convex solutions of (\ref{e1.3}) and (\ref{e1.4}), then $u$ is unique up to a constant.
\end{lemma}

Next, by Lemma 5.2 in \cite{CHB}, we obtain
\begin{lemma}\label{yinli2.4.2}
If  $u\in C^{2}(\bar{\Omega})$ is a uniformly convex solution of (\ref{e1.3}) and (\ref{e1.4}), then $u\in C^{\infty}(\bar{\Omega})$.
\end{lemma}

By the continuity method, we now show the following \\
\noindent{\bf Proof of Theorem \ref{t1.1}.}
For each $t\in [0,1]$, set
$$G^{t}(Du, D^{2}u)=tG(Du, D^{2}u)+(1-t)\sum_{i=1}^{n}\arctan \lambda_{i}(D^{2}u).$$
Define
$$\hat{F}=\sum_{i=1}^{n}\arctan \lambda_{i}(D^{2}u).$$
Note that since $u_{kl}=\upsilon b_{ki}a_{ij}b_{jl}$, then
$$\hat{F}_{ij}=\frac{\partial\hat{F}}{\partial a_{ij}}=\frac{\partial\hat{F}}{\partial u_{kl}}\frac{\partial u_{kl}}{\partial a_{ij}}=\upsilon b_{ki}\frac{\partial\hat{F}}{\partial u_{kl}}b_{jl}$$
and
$$\hat{F}_{ij,rs}=\frac{\partial^{2}\hat{F}}{\partial a_{ij}\partial a_{rs}}=\upsilon^{2} b_{ki}b_{pr}\frac{\partial^{2}\hat{F}}{\partial u_{kl}\partial u_{pq}}b_{qs}b_{jl}.$$
By the virtue of $b_{ij}$ being positive matrix, we see that the structure conditions operator $G_{t}$ is similar to $G$ as we discussed in section 2 for uniformly $t\in [0,1]$.
One can  consider
\begin{equation}\label{e5.1}
\left\{ \begin{aligned}G^{t}(Du, D^{2}u)&=c(t), \ \ x\in \Omega, \\
Du(\Omega)&=\tilde{\Omega},
\end{aligned} \right.
\end{equation}
which is equivalent to
\begin{equation}\label{e5.2}
\left\{ \begin{aligned}G^{t}(Du, D^{2}u)&=c(t), \ \ &&x\in \Omega, \\
h(Du)&=0,\ \ &&x\in\partial \Omega.
\end{aligned} \right.
\end{equation}
By Brendle-Warren's Theorem in \cite{SM} , (\ref{e5.1}) is solvable if $t=0$. According to Lemma \ref{yinli2.4.2}, we define the closed subset
$$\mathfrak{B_{1}}:=\{u\in C^{2,\alpha}(\bar{\Omega}): \int_{\Omega}u=0\}$$
in $C^{2,\alpha}(\bar{\Omega})$ and
$$\mathfrak{B_{2}}:= C^{\alpha}(\bar{\Omega})\times C^{1,\alpha}(\partial\Omega).$$
Moreover, we define a map from $\mathfrak{B_{1}}\times \mathbb{R}$ to $\mathfrak{B_{2}}$ by
$$\mathfrak{F}^{t}:=(G^{t}(Du, D^{2}u)-c(t),h(Du)).$$
Then the linearized operator $D\mathfrak{F}^{t}_{(u,c)}:\mathfrak{B_{1}}\times \mathbb{R}\rightarrow \mathfrak{B_{2}}$ is given by
$$D\mathfrak{F}^{t}_{(u,c)}(w,a)=G^{t}_{ij}(Du, D^{2}u)\partial_{ij}w-G^{t}_{p_{i}}(Du, D^{2}u)\partial_{i}w-a, h_{p_i}(Du)\partial_i w).$$
Following the same proof as Proposition 3.1 in \cite{SM}, we obtain that $D\mathfrak{F}^{t}_{(u,c)}$ is invertible for any $(u,c)$ satisfying (\ref{e5.1}) and $t\in[0,1]$.

Write
$$I:=\left\{t\in[0,1]:\text{(\ref{e4.1}) has at least one convex solution}\right\}.$$
Since $0\in I$,  $I$ is not empty. We claim that $I=[0,1]$, which is equivalent to the fact that $I$ is not only open, but also closed. It follows from Proposition 3.1 in \cite{SM} again and Theorem 17.6 in \cite{GT} that $I$ is open. So we only need to prove that $I$ is a closed subset of $[0,1]$.

That $I$ is closed is equivalent to the fact that for any sequence $\{t_k\}\subset I$, if $\lim_{k\rightarrow \infty}t_k=t_0$, then $t_0\in I$. For $t_k$, denote $(u_k, c(t_k))$ solving
\begin{equation*}
\left\{ \begin{aligned}G^{t_{k}}(Du_{k}, D^{2}u_{k})&=c(t_k), \ \ x\in \Omega, \\
Du_k(\Omega)&=\tilde{\Omega}.
\end{aligned} \right.
\end{equation*}
It follows from Lemma \ref{lem4.12} and the proof of Lemma 5.2 in \cite{CHB} that $\|u_k\|_{C^{2,\alpha}(\bar{\Omega})}\leq C$, where $C$ is independent of $t_k$. Since
$$|c(t)|=|G^{t}(Du, D^{2}u)|\leq n\pi,$$
by Arzela-Ascoli Theorem we know that there exists $\hat{u}\in C^{2,\alpha}(\bar{\Omega})$, $\hat{c}\in\mathbb{R}$ and a subsequence of $\{t_k\}$, which is still denoted as $\{t_k\}$, such that letting $k\rightarrow \infty$,
\begin{equation*}
\left\{ \begin{aligned}
&\left\|u_k-\hat{u}\right\|_{C^{2}(\bar{\Omega})}\rightarrow 0, \\
&c(t_k)\rightarrow \hat{c}.
\end{aligned} \right.
\end{equation*}
Since $(u_k, c(t_k))$ satisfies
\begin{equation*}
\left\{ \begin{aligned}G^{t_{k}}(Du_{k}, D^{2}u_{k})&=c(t_k), \ \ &&x\in \Omega, \\
h(Du_k)&=0,\ \ &&x\in\partial \Omega,
\end{aligned} \right.
\end{equation*}
letting $k\rightarrow \infty$, we arrive at
\begin{equation*}
\left\{ \begin{aligned}G^{t_{0}}(D\hat{u}, D^{2}\hat{u})&=\hat{c}, \ \ &&x\in \Omega, \\
h(D\hat{u})&=0,\ \ &&x\in\partial \Omega.
\end{aligned} \right.
\end{equation*}
Therefore, $t_0\in I$, and thus $I$ is closed. Consequently, $I=[0,1]$. By Lemma \ref{yinli2.4.1} we know that the solution of (\ref{e5.1}) is unique up to a constant.

Then we complete the proof of Theorem \ref{t1.1}.
\qed

{\bf Acknowledgments:} The authors thank Professor Qianzhong Ou and Dr.Nianhong Zhou for helpful discussions and suggestions related to this work. The authors also would like to thank  referees for useful comments, which improve the paper.


\begin{thebibliography}{DU}

 \bibitem{HO}R.L. Huang, Q.Z. Ou, {\it On the Second Boundary Value Problem for a Class of Fully Nonlinear Equations,}
 J. Geom. Anal. {\bf 27} (2017), 2601--2617.
 \bibitem{CHB} Chong Wang, Rongli Huang, Jiguang Bao, {\it
 On the second boundary value problem for Lagrangian mean curvature equation,}arXiv:1808.01139.

 \bibitem{MW} M. Warren, {\it Calibrations associated to Monge-Amp$\grave{e}$re equations,} Trans. Amer. Math. Soc. {\bf 362} (2010), 3947--3962.

 \bibitem{HL}  R. Harvey, H.B. Lawson, {\it Calibrated geometry}.
Acta Math. {\bf 148} (1982), 47--157.

\bibitem{Mealy} J. Mealy, {Volume maximization in semi-Riemannian manifolds,} Indiana Univ. Math. J. {\bf 40} (1991) 793--814.

\bibitem{SM}S. Brendle, M. Warren,
{\it A boundary value problem for minimal Lagrangian graphs,} J. Differential Geom. {\bf 84} (2010), 267--287.

\bibitem{HR} R.L. Huang, {\it On the second boundary value problem for Lagrangian mean curvature flow,} J. Funct. Anal. {\bf 269} (2015), 1095--1114.

\bibitem{P}P. Delanoe, {\it Classical solvability in dimension two of the second boundary-value problem associated with the Monge-Amp$\grave{e}$re operator,} Ann. Inst. H. Poincar$\acute{e}$ Anal. Non Lin$\acute{e}$aire.  {\bf 8} (1991), 443--457.

\bibitem{L} L. Caffarelli, {\it Boundary regularity of maps with convex potentials,} II, Ann. of Math. Stud.
{\bf 144} (1996), 453--496.

 \bibitem{JU} J. Urbas, {\it On the second boundary value problems for equations of Monge-Amp\`{e}re type,} J. Reine Angew. Math. {\bf 487} (1997), 115--124.
 \bibitem{OK} O.C. Schn$\ddot{\text{u}}$rer, K. Smoczyk, {\it Neumann and second boundary value problems for Hessian and  Gauss curvature flows,}
Ann. Inst. H. Poincar$\acute{e}$ Anal. Non Lin$\acute{e}$aire. {\bf 20} (2003), 1043--1073.


\bibitem{HRY} R.L. Huang, Y.H. Ye, {\it On the second boundary value problem for a class of fully nonlinear flows} I, Int. Math. Res. Not. IMRN {\bf 18}(2019),  5539-5576.


\bibitem{CHY} J.J. Chen, R.L. Huang, Y.H. Ye, {\it On the second boundary value problem for a class of fully nonlinear flows} II, Archiv Der Mathematik. {\bf 111}(2018), 1-13.

\bibitem{G}
 G. Smith, {\it Special Lagrangian curvature},  Math. Ann. {\bf 355} (2013), no. 1, 57-95.

\bibitem{HL1}  R. Harvey, H.B. Lawson, {\it Pseudoconvexity for the special Lagrangian potential equation},
 Calc. Var. Partial Differential Equations. {\bf60} (2021), 39 pp.

\bibitem{J}
J. Urbas,  {\it Weingarten hypersurfaces with prescribed gradient image},  Math. Z.   {\bf240} (2002), no. 1, 53-82.


\bibitem{GT} D. Gilbarg, N.S. Trudinger, {\it Elliptic Partial Differential Equations of Second Order},
Springer-Verlag, Berlin, 2nd ed., 1998.


\bibitem{LLJ}
L. Caffarelli, L. Nirenberg, and  J. Spruck, {\it Nonlinear  second  order  equations IV.Starshaped
conipuct  Weingurten hypersurfuces. Current topics in purtiul differential equations},  ed. by Y.Ohya,K.Kasahara,
N.Shimakura, 1986, pp. 1-26,  Kinokunize Co.,Tokyo.


\bibitem{J1}
J. Urbas, {\it Nonlinear oblique boundary value problems for two dimensional curvature
equations},  Advances in Differ. Eqns. {\bf240} (1996), no. 1, 301-336.

\bibitem{J2}
J. Urbas, {\it An interior curvature bound for hypersurfaces of prescribed k-th mean curvature}, J. Reine Angew. Math. {\bf 519}(2000), 41-57.































































\end{thebibliography}
\end{document}